\documentclass[10pt]{amsart}
\usepackage{amsmath,amssymb,latexsym,soul,cite,mathrsfs,amsthm}

\usepackage{color,enumitem,graphicx}
\usepackage[colorlinks=true,urlcolor=blue,
citecolor=red,linkcolor=blue,linktocpage,pdfpagelabels,
bookmarksnumbered,bookmarksopen]{hyperref}
\usepackage[english]{babel}

\usepackage[left=2.9cm,right=2.9cm,top=2.8cm,bottom=2.8cm]{geometry}
\usepackage[hyperpageref]{backref}

\usepackage[colorinlistoftodos]{todonotes}
\makeatletter
\providecommand\@dotsep{5}
\def\listtodoname{List of Todos}
\def\listoftodos{\@starttoc{tdo}\listtodoname}
\makeatother

\numberwithin{equation}{section}

\def\cal{\mathcal}
\newtheorem{lemma}{Lemma}
\newtheorem{proposition}{Proposition}
\newtheorem{theorem}{Theorem}
\newtheorem*{theorem*}{Theorem}

\newtheorem{remark}{Remark}

\def\cal{\mathcal}
\def\H{H^1_0(\Omega)}

\title[positive solutions for  the fractional laplacian ]
{positive solutions for  the fractional laplacian  in the almost critical case
in a bounded domain}

\author[G. M. Figueiredo]{Giovany M. Figueiredo}
\author[G. Siciliano]{Gaetano Siciliano}

\address[G. M. Figueiredo]{\newline\indent Faculdade de Matem\'atica
\newline\indent 
Universidade Federal do Par\'a
\newline\indent
66075-110, Bel\'em - PA, Brazil}
\email{\href{mailto:giovany@ufpa.br}{giovany@ufpa.br}}

\address[G. Siciliano]{\newline\indent Departamento de Matem\'atica
\newline\indent 
 Universidade de S\~ao Paulo 
\newline\indent 
Rua do Mat\~ao 1010,  05508-090 S\~ao Paulo, SP, Brazil }
\email{\href{mailto:sicilian@ime.usp.br}{sicilian@ime.usp.br}}

\thanks{Giovany M. Figueiredo was partially
supported by  CNPq, Brazil. Gaetano Siciliano  was partially supported by
Fapesp and CNPq, Brazil. }

\begin{document}

\maketitle

\begin{abstract}
We prove existence of multiple   positive solutions for a {\sl fractional scalar field equation}
in a bounded domain, whenever $p$ tends to the critical Sobolev exponent.
By means of the ``photography method'', we prove that
the topology of the domain  furnishes a lower bound
on the number of positive solutions.

\end{abstract}

\maketitle

\section{Introduction}

In  the celebrated papers \cite{BC,BC2} Benci, Cerami and Passaseo proved an existence result 
of positive solutions of the following problem
\begin{equation}
\label{model}\left\{
\begin{array}
[c]{ll}
-\Delta u +u=|u|^{p-2}u & \quad\text{ in } \Omega\\
u=0 & \quad\text{ on } \partial\Omega
\end{array}
\right.
\end{equation}
where $\Omega\subset\mathbb{R}^{N}$ is a smooth and bounded domain, $N\ge3$
and $p<2^{*}=\frac{2N}{N-2}$, the critical Sobolev exponent of the embedding
of $\H$ in the Lebesgue spaces. 
Roughly speaking they show that (among other results), for $p$ near $2^{*}$, the number of positive solutions
is bounded below by a topological invariant associated to $\Omega$.
More specifically they prove the following
\begin{theorem*}
\label{ThA} There exists a $\bar{p}\in(2,2^{*})$ such that for every
$p\in[\bar p, 2^{*})$ problem \eqref{model} has (at least) $\emph{cat}_{\bar\Omega}\,(\bar\Omega)$
positive solutions. Even more, if $\Omega$ is not contractible in itself, the number of solutions is
$\emph{cat}_{\bar\Omega}\,(\bar\Omega)+1.$
\end{theorem*}

Hereafter given a topological pair $A\subset X, \textrm{cat}_{X}(A)$ is the Ljusternik-Schnirelmann category of the set $A$ in $X$ (see e.g. \cite{J}).

To prove this result, the authors used variational methods: an energy functional related to the problem
is introduced in such a  way that the solutions are seen as critical point of this functional restricted to  $L^{p}-$ball.
Then the ``photography method''
(which permits to see a photography of the domain $\Omega$ in a suitable sublevel of the   functional) is implemented
in order to prove the existence of many critical points by means of the classical  Ljusternick-Schnirelmann Theory.

\bigskip

The aim of this paper is to prove the fractional counterpart of  the above Theorem. Indeed,
due to the large literature appearing in these last years on fractional operators,
 it is very natural to ask if 
a similar result also holds for the fractional laplacian.
 In other words we consider in this paper
the following nonlocal problem
\begin{equation}\label{eq:FBC}
\left\{
\begin{array}
[c]{ll}
(- \Delta )^{s}u + u =|u|^{p-2}u & \text{ in } \Omega,\\
u=0 & \text{ on } \mathbb R^{N}\setminus\Omega,
\end{array}
\right.
\end{equation}
where $s\in(0,1)$, 
$p\in(2,2_{s}^{*})$ with $2_{s}^{*}:=2N/(N-2s), N>2s$.

 The operator $(-\Delta)^{s}$ is the {\sl fractional Laplacian}
which is  defined by
$$
(-\Delta)^{s}u (x) := C(N,s)\lim_{ \varepsilon\to 0^{+}} \int_{\mathbb R^{N}\setminus B_{\varepsilon}(x)}
\frac{u(x)-u(y)}{|x-y|^{N+2s}}dy\,,  \qquad x\in \mathbb R^{N}.$$
for a suitable constant $C(N,s)>0$ whose exact value is not really important for our purpose.
The Dirichlet condition in \eqref{eq:FBC} is then given on $\mathbb R^{N}\setminus\Omega$ reflecting the fact that
 $(-\Delta)^{s}$ is a nonlocal operator.

Before to state our result, let us introduce few basic notations.
For a measurable function $u:\mathbb R^{N}\to \mathbb R$ let 
\begin{equation*}\label{eq:gagliardo}
[u]^{2}_{D^{s,2}(\mathbb R^{N})}:=\int_{\mathbb R^{2N}} \frac{|u(x)-u(y)|^{2}}{|x-y|^{N+2s}} dxdy
\end{equation*}
be  the (squared) {\sl Gagliardo seminorm} of $u$. Let us define the Hilbert space
 $$D^{s,2}(\mathbb R^{N})=\left\{u\in L^{2_{s}^{*}}(\mathbb R^{N}): [u]^{2}_{D^{s,2}(\mathbb R^{N})}<+\infty\right\},$$
which is continuously embedded into $L^{2_{s}^{*}}(\mathbb R^{N})$. Let finally
 $$D_{0}^{s,2}(\Omega)=\left\{u\in D^{s,2}(\mathbb R^{N}): u\equiv 0 \text{ in } \ \mathbb R^{N}\setminus \Omega\right\}.$$

Form now on it will be convenient to
adopt the following convention: functions defined in a subset of $\mathbb R^{N}$, let us say $A$,  will be thought extended by zero
on $\mathbb R^{N}\setminus A$, whenever regarded as functions defined on the whole $\mathbb R^{N}$.

Note that being $\partial\Omega$ smooth, $D_{0}^{s,2}(\Omega)$ can be also defined as the completion of
$C^{\infty}_{0}(\Omega)$ under the norm $[\cdot]_{D^{s,2}(\mathbb R^{N})}.$
Moreover, 
it is $D^{s,2}_{0}(\Omega) = \{u\in H^{s}(\mathbb R^{N}): u\equiv 0 \text{ on }
\mathbb R^{N}\setminus\Omega\}$. 

Recall that we have the continuous embedding $D^{s,2}_{0}(\Omega) \hookrightarrow L^{p}(\Omega)$
for $1\leq p\leq 2_{s}^{*}$ and that the embedding is compact for $1\leq p< 2^{*}_{s}.$

%

\medskip

We then say that $u\in D_{0}^{s,2}(\Omega)$ is  a solution (in the distributional sense)
of \eqref{eq:FBC} if
\begin{equation}\label{eq:solufrac}
\forall \,v\in D_{0}^{s,2}(\Omega) :\quad \int_{\mathbb R^{2N}} \frac{(u(x)-u(y)) (v(x) - v(y))}{|x-y|^{N+2s}}dxdy+\int_{\mathbb R^{N}}u v dx=\int_{\mathbb R^{N}} |u|^{p-2}uv dx.
\end{equation}

The main result of the paper gives a positive answer on the possibility of extending the Benci, Cerami and Passaseo 
result to the fractional case.
\begin{theorem}\label{th:th}
For $s\in(0,1),N>2s$,
there exists a $\bar{p}\in(2,2_{s}^{*})$ such that for every
$p\in[\bar p, 2_{s}^{*})$ problem \eqref{eq:FBC} possesses (at least)
$\emph{cat}_{\bar\Omega}\,(\bar\Omega)$ positive solutions.
Whenever $\Omega$ is not contractible in itself,  the number of solutions is 
$\emph{cat}_{\bar\Omega}\,(\bar\Omega)+1.$
\end{theorem}

Beside the case when $p\to 2^{*}$, the papers  \cite{BC,BC2} treat also when a certain parameter $\lambda$ appearing in the equation tends to $+\infty$. We do not enter in details here, but the same type of result is obtained: for large $\lambda$ the domain topology gives a lower bound on the number of positive solutions of the problem.
 However it is readily seen that this
last case can be equivalently reformulated as a problem in an expanding domain, simply by a change of variables
which ``transfer'' the parameter $\lambda$ from the equation to the domain.
In the same spirit, in \cite{BC3} the influence of the domain topology is studied for
semiclassical equations, that is, roughly speaking, when a parameter $\varepsilon$ which appears in the equation
tends to zero.

Subsequently, after the papers  \cite{BC, BC2,BC3}, many
authors have used the same methods 
to prove multiplicity results of solutions (depending on the domain topology)
whenever  $\lambda\to+\infty$, that is for problems in expanding domains, or $\varepsilon\to0$,
that is for semiclassical states.

Nevertheless, to the best of our knowledge, there is only another paper in the literature
dealing with the case in which  the role of parameter is taken by the exponent of the nonlinearity 
which tends to the critical Sobolev exponent:
see \cite{Sicilia} where the Schr\"odinger-Poisson system is studied. 
We think that, even if less explored,  the case in which the
parameter is the exponent of the power nonlinearity is equally interesting;  our goal is then
to give a contribution in this direction.
Observe finally, that  by considering $s=1$, our proof can be adapted to recover
the result of \cite{BC} by using the method of the Nehari manifold, in place of the $L^{p}-$ball
as done in the paper of Benci and Cerami.
\medskip

The paper is organized in the following way.

In Section \ref{sec:Variational} we give the variational setting in which  problem  \eqref{eq:FBC} is settled.
Section \ref{sec:limit} deals  with a related limit problem, which will be usefull in order to prove Theorem \ref{th:th}.
Finally, in Section \ref{sec:prova} after introducing the barycenter map and prove some important properties,
the proof of  Theorem \ref{th:th} is given.

\medskip

Let us finish this section with basic notations that will be used in all the paper.
\subsection*{Notations}
Without loss of generality we assume in all the paper $0\in\Omega$. We denote by
$|\,.\,|_{L^{p}(A)}$ the $L^{p}-$norm of a function defined on the domain $A$.
If the domain is  $\Omega$ or $\mathbb R^{N}$ (it should be clear from the context) 
we will use the notation $|\,.\,|_{p}$. 

We use $B_{r}(y)$ for the closed ball of radius $r>0$ centered in $y$. If $y=0$
we simply write $B_r.$

The letter $c$ will be used indiscriminately
to denote a suitable positive constant whose value may change from line to line 
and we will use $o(1)$ for a quantity which goes to zero.

Other notations will be introduced whenever we need.

\section{the variational setting}\label{sec:Variational}
It is easily seen that a solution in the sense \eqref{eq:solufrac}
 of problem \eqref{eq:FBC}  can be found as  a critical point of the $C^{1}$
 functional
\begin{equation}\label{eq:Ip}
I_{p}(u)=\frac{1}{2}[u]^{2}_{D^{s,2}(\mathbb R^{N})}
+\frac12|u|_{2}^{2}
-\frac{1}{p}|u|_{p}^{p}
  \qquad u\in D_{0}^{s,2}(\Omega).
\end{equation}

Observe that, for $u\in D^{s,2}_{0}(\Omega)$ we can write equivalently
$$
[u]^{2}_{D^{s,2}(\mathbb R^{N})} =
 \int_{\mathbb R^{2N}\setminus(\Omega^{c} \times \Omega^{c})}\frac{|u(x)-u(y)|^{2}}{|x-y|^{N+2s}} dx dy$$
where of course $\Omega^{c}:=\mathbb R^{N}\setminus \Omega$.
We will use the next result
\begin{lemma}\cite[Lemma 6]{SV}
If $u\in D_{0}^{s,2}(\Omega)$ then 
\begin{equation*}
|u|_{L^{2^{*}_{s}}(\Omega)} \leq c \,[u]^{2}_{D^{s,2}(\mathbb R^{N})}
\end{equation*}
for a suitable constant $c>0$. In particular it follows that 
$$\|u\|^{2}: = [u ]_{D^{s,2}(\mathbb R^{N})}^{2} +|  u |_{2}^{2}$$ 
gives an equivalent (squared) norm on $D_{0}^{s,2}(\Omega)$.
\end{lemma}

Then we can write
$$I_{p}(u) =\frac12\|u\|^{2} - \frac1p |u|_{p}^{p}.$$

\medskip

A fundamental tool in order to apply variational techniques is the
so-called \emph{Palais-Smale condition} (PS for brevity).
If $M$ is a smooth manifold in $D^{s,2}_{0}(\Omega)$, we say that $I_{p}$
satisfies the PS condition on $M$ (or restricted to $M$) if  every sequence $\{u_{n}\}\subset M$
such that
\begin{equation}\label{PS}
\{I_{p}(u_{n})\} \ \text{ is bounded \ \  and }\ \ I_{p}^{\prime}(u_{n})\rightarrow0 \ \ \text{ in } D^{-s,2}(\Omega),
\end{equation} 
admits a converging subsequence. 
Clearly $I_{p}'(u_{n})$, has to be intended as the tangencial component of $I_{p}'(u_{n})$ to $T_{u_{n}}M$.
Sequences which satisfy
\eqref{PS} are called \emph{Palais-Smale sequences.}


\medskip

To prove the theorem we use the general ideas of Benci, Cerami and Passaseo adapting 
their arguments to our problem which contains a nonlocal operator.

A natural way of finding the critical points of $I_{p}$ which is unbounded above and below,
is to restrict the functional to a suitable manifold, {\sl the Nehari manifold}, 
on which it results bounded below and hence the
classical Ljusternick-Schnirelmann Theory can be employed.

\subsection{The Nehari manifold}

In this subsection we recall some known facts 
about the Nehari manifold  that will be used throughout the paper. %

The Nehari manifold associated to \eqref{eq:Ip} is defined
by
\begin{equation*}
\label{Nehari}{\mathcal{N}}_{p}=\left\{ u\in D_{0}^{s,2}(\Omega)\setminus\{0\}:
G_{p}(u)=0\right\}
\end{equation*}
where
$$
G_{p}(u):=I_{p}^{\prime}(u)[u]=[u]^{2}_{D^{s,2}(\mathbb R^{N})}+ |u|_{2}^{2} -|u|_{p}^{p}\,.
$$
Note that on $\mathcal N_{p}$  the functional \eqref{eq:Ip} takes the form
\begin{equation}\label{eq:Ivinc}
I_{p}(u)=\frac{p-2}{2p}\| u\| ^{2}\geq0.
\end{equation}
Sometimes we will refer to \eqref{eq:Ivinc} as the constraint functional, also
denoted with $I_{p}|_{\mathcal{N}_{p}}$. 
In the next Lemma we list the basic properties of the Nehari manifold, easy to check
in a standard way.
\begin{lemma}\label{lemmanehari}
For $p\in (2,2^{*}_{s}]$, we have
\begin{itemize}
\item[1.] $\mathcal{N}_{p}$ is a $C^{1}$ manifold\,,

\item[2.] there exists $c>0$ such that for every $u\in \mathcal{N}_{p}:c\le\|u\|\,, $

\item[3.] for any $u\neq0$ there exists a unique $t_{u}>0$ such that
$t_{u}u\in\mathcal{N}_{p}$ and $\inf_{0\neq u\in D^{s,2}_{0}(\Omega)} t_{u} >0$\,,

\item[4.] the following equalities are true
\[
m_{p}:=\inf_{u\neq0}\max_{t>0}I_{p}(tu)=\inf_{g\in\Gamma_{p}} \, \max
_{t\in[0,1]} I_{p}(g(t))=\inf_{u\in\mathcal{N}_{p}}I_{p}(u)>0\,
\]
where
$$
\Gamma_{p}=\left\{g\in C([0,1];D_{0}^{s,2}(\Omega)) : g(0)=0, I_{p}(g(1))\le0,
g(1)\neq0\right\}.
$$
\end{itemize}
\end{lemma}

Moreover the manifold ${\mathcal{N}}_{p}$ is a {\sl natural constraint} for $I_{ p}$  in the
sense that any 
critical point of $I_{p}$ restricted to ${\mathcal{N}
_{p}}$ is also a critical point for the ``free'' functional $I_{p}$ in the whole Hilbert space.
Hence the
(constraint) critical points we find are solutions of our problem since no
Lagrange multipliers will appear.

With a standard proof, one show that the Nehari manifold well-behaves with respect to the PS sequences, that is
\begin{lemma}\label{lemmaPS}
Let $\{u_{n}\}\subset\mathcal{N}_{p}$ be a PS sequence
for $I_{p}|_{\mathcal{N}_{p}}$. Then it is a PS
sequence for the free functional $I_{p}$ on the whole space $D^{s,2}_{0}(\Omega)$.
Moreover, if $p\in (2,2^{*}_{s})$ then $ I_{p}$ restricted to $\mathcal N_{p}$ satisfies the PS condition.
\end{lemma}

%

As a consequence we set 
\begin{equation*}\label{mp}
\forall\,p\in(2,2_{s}^{*})\,:\ \ m_{ p}:=\min_{\mathcal{N}_{ p}} I_{p}=I_{ p}(u_{ p})\,,
\end{equation*}
i.e. $m_{ p}$ is achieved on a function (also called {\sl ground state}), hereafter denoted with $u_{ p}\in \cal N_p$. 
Observe that the family of minimizers $\{u_{ p}\}_{p\in(2,2_{s}^{*})}$ is bounded away from
zero; indeed, since $u_{ p}\in\mathcal{N}_{ p}\,$,
\begin{equation}
\label{dis}\|u_{ p}\|^{2}\le|u_{ p}|_{p}^{p}\le C\|u_{ p}\|^{p}
\end{equation}
where the positive constant $C$ is  independent of $p$.
Hence
\begin{equation*}
\exists\,c>0 \ \ \mbox{ s.t. } \forall\, p\in(2,2_{s}^{*})\, :\ \ 0< c\le\|u_{
p}\|.
\end{equation*}

\begin{remark}
\label{rem} By \eqref{dis}, we deduce that $\{|u_{ p}|_{p}\}_{p\in(2,2_{s}^{*})}$ 
is also bounded away from zero. Moreover, 
 the H\"{o}lder inequality implies
\[
|u_{ p}|_{p}\le|\Omega|^{\frac{2_{s}^{*}-p}{2_{s}^{*}p}}|u_{ p}|_{2_{s}^{*}}%
\]
 so that also $\{|u_{ p}|_{2_{s}^{*}}\}_{p\in(2,2_{s}^{*})}$ is far away from zero.
\end{remark}

It will be important for us, in order to prove the main Theorem (see subsection \ref{subsec:finale}), to evaluate the limit of the ground state levels
$m_{p}$ for $p$ tending to $2^{*}_{s}$.



\section{The limit problem}\label{sec:limit}
With the aim of evaluating the
limit of the sequence $\{m_p\}_{p\in(2,2_{s}^*)}$
when $p\rightarrow2_{s}^*$,
we start by considering  a limit problem related to \eqref{eq:FBC}. 
Let us introduce the $C^{1}$ functional on $D^{s,2}_{0}(\Omega)$
\begin{equation}\label{eq:I*}
I_{*}(u)=\frac{1}{2}\|u\|^{2}-\frac{1}{2_{s}^{*}}|u|_{2_{s}^{*}}^{2_{s}^{*}}\,
\end{equation}
whose critical points are the solutions of
\begin{equation}
\left\{
\begin{array}
[c]{ll}%
\label{star} (-\Delta)^{s}u +u=|u|^{2_{s}^{*}-2}u & \quad\text{ in } \Omega\\
u=0 & \quad\text{ on } \mathbb R^{N}\setminus \Omega.
\end{array}
\right.
\end{equation}
The lack of compactness of the embedding of $D^{s,2}_{0}(\Omega)$
in $L^{2_{s}^{*}}(\Omega)$ implies that 
$I_{*}$ does not satisfies the PS condition at every level; indeed 
the {\sl fractional conformal scaling }
$$u(\cdot)\longmapsto u_{R}(\cdot):=R^{(N-2s)/2}u(R(\cdot))\,, \ \ \ \ R>1$$
 leaves invariant the $[u]_{D^{s,2}(\mathbb R^{N})} $ and the $L^{2_{s}^*}-$norm of $u:\Omega\to \mathbb R$:
$$[u_{R}]_{D^{s,2}(\mathbb R^{N})}^{2}=[u]_{D^{s,2}(\mathbb R^{N})}^{2} \qquad  \ 
|u_{R}|_{2_{s}^{*}}^{2_{s}^{*}}=|u|_{2_{s}^{*}}^{2_{s}^{*}}\,.$$
On the other hand, for $p\in[1,2^{*}_{s})$,
$$ |u_{R}|_{p}^{p}=R^{\frac{p(N-2s)-2N}{2}}|u|_{p}^{p}\longrightarrow 0\quad 
\text{as } R\longrightarrow +\infty $$
and then it is easy to see that in a bounded domain $\Omega$
the infimum
$$S := \inf_{0\neq u\in D^{s,2}_{0}(\Omega)}\frac{\|u\|^{2}}{|u|_{2^{*}_{s}}^{2}}$$
is never achieved. Let also 
$$
\mathcal{N}_{*}=\{u\in D^{s,2}_{0}(\Omega): G_{*}(u)=0\}\,, \ \ \ G_{*}
(u)=\|u\|^{2}-|u|_{2_{s}^{*}}^{2_{s}^{*}}%
$$
be the Nehari manifold associated to problem \eqref{star}
and
$$
m_{*}:=\inf_{\mathcal{N}_{*}} I_{*}= \inf_{u\neq0}\max_{t>0} I_{*}(tu).
$$
If $u\in \mathcal N_{*}$ then $I_{*}(u) = \frac sN\|u\|^{2}$.

The following lemma is probably known but for the sake of completeness we give the  proof.
\begin{lemma}
There holds
$$
m_{*}=\frac{s}{N}S^{N/2s}$$
where $S$ is the 
is the best  (fractional) Sobolev constant defined above.
\end{lemma}
\begin{proof}
For $A,B>0$ it results
\[
\max_{t>0} \left\{ \frac{t^{2}}{2}A-\frac{t^{2_{s}^{*}}}{2_{s}^{*}}B\right\} =\frac
{s}{N}\left( \frac{A}{B^{2/2^{*}_{s}}}\right) ^{N/2s}.
\]
Then
\[
m_{*}=\inf_{u\neq0}\max_{t>0} I_{*}(tu)=\frac{s}{N}\left( \inf_{u\neq0}
\frac{\|u\|^{2}}{|u|_{2_{s}^{*}}^{2}}\right) ^{N/2s} =\frac{s}{N}S^{N/2s}.
\]
\end{proof}
In particular it is easy to see that $m_{*}$ is not achieved. 

\medskip

As a first step, we show that $m_*$ is  an upper bound for the sequence of ground states levels $\{m_p\}_{p\in(2,2_{s}^*)}$.

\begin{lemma}\label{limitate}
We have
$$\limsup_{p\rightarrow2_{s}^{*}} m_{p}\le m_{*}.$$
\end{lemma}
\begin{proof}
Given $\varepsilon>0$, 
there exists $u\in
\mathcal{N}_{*}$ such that
\begin{equation}
\label{e/2}I_{*}(u)
=\frac{s}{N}\|u\|^{2}
<m_{*}+\varepsilon.
\end{equation}
Now consider, for any $p\in(2,2_{s}^{*})$, the unique positive value $t_{p}$ such
that $t_{p} u_{R}\in\mathcal{N}_{p}\,.$ 
By definition, $t_{p}$ satisfies
\begin{equation}
\label{rel}\|t_{p} u_{R}\|^{2}
=|t_{p} u_{R}|_{p}^{p}
\end{equation}
from which we deduce:

\begin{itemize}
\item $\{t_{p}\}_{p\in(2,2_{s}^{*})}$ is bounded away from zero.
\end{itemize}
Indeed by \eqref{rel} and the embedding of $D_{0}^{s,2}(\Omega)$ in $L^{p}(\Omega)$   we get
$\|t_{p} u_{R}\|^{2}\le C\|t_{p} u_{R}\|^{p}$ so $\|t_{p} u_{R}\|^{2}\ge c$
and finally $t_{p}^{2}\ge\frac{c}{\|u_{R}\|^{2}}\ge\frac{c}{\|u\|^{2}}>0.$

\begin{itemize}
\item $\{t_{p}\}_{p\in(2,2_{s}^{*})}$ is bounded above.
\end{itemize}
Indeed $$\|u_{R}\|^{2}
=t_{p}^{p-2}|u_{R}|_{p}^{p}$$ and, by the continuity of the map
$p\mapsto|u_{R}|_{p}$\,, it is readily seen that if $t_{p}$ tends to
$+\infty$ as $p\to 2_{s}^{*}$ we get a contradiction.

\smallskip

So we may assume that $\lim_{p\rightarrow2_{s}^{*}} t_{p}=t_{*}$, and passing to
the limit in \eqref{rel} we get
\begin{align*}
t_{*}^{2}[u]_{D^{s,2}(\mathbb R^{N})}^{2}+\frac{t_{*}^{2}}{R^{2s}}|u|_{2}^{2}
=t_{*}^{2_{s}^{*}%
}|u|_{2_{s}^{*}}^{2_{s}^{*}}
 =t_{*}^{2_{s}^{*}} \|u\|^{2}
\end{align*}
or equivalently,
\[
(t_{*}^{2_{s}^{*}}-t_{*}^{2})[u]_{D^{s,2}(\mathbb R^{N})}^{2}
= (\frac{t_{*}^{2}}{R^{2s}}-t_{*}^{2_{s}^{*}}) |u|_{2}^{2}
.
\]
Now if $R$ is chosen sufficiently large, the r.h.s. above is negative and 
we deduce
\begin{equation}
\label{t}t_{*}<1.
\end{equation}
Furthermore
\begin{eqnarray*}
I_{p}(t_{p} u_{R})=\frac{p-2}{2p}\|t_{p} u_{R}\|^{2}
=\frac{p-2}{2p}t^{2}_{p}[u]^{2}_{D^{s,2}(\mathbb R^{N})}
+\frac{p-2}{2p}\frac{t_{p}^{2}}{R^{2s}}|u|_{2}^{2}
\end{eqnarray*}
and passing to the limit for $p\rightarrow2_{s}^*$, taking advantage of \eqref{t},
\begin{equation*}
\lim_{p\rightarrow2_{s}^{*}} I_{p}(t_{p} u_{R}) 
=\frac{s}{N}t_{*}^{2}[u]_{D^{s,2}(\mathbb R^{N})}^{2}+\frac{s}{N}\frac{t_{*}^{2}}{R^{2s}}|u|_{2}^{2}
 <\frac{s}{N}\| u\|^{2}
\end{equation*}
Lastly
we get, using \eqref{e/2},
\[
\limsup_{p\rightarrow2_{s}^{*}} m_{p}\le\lim_{p\rightarrow2_{s}^{*}} I_{p} (t_{p}
u_{R})<\frac{s}{N}\|u\|^{2} <m_{*}+\varepsilon
\]
which concludes the proof since $\varepsilon$ is arbitrary.
\end{proof}

Recalling \eqref{eq:Ivinc}, the boundedness of 
$\{m_{p}\}_{p\in(2,2_{s}^{*})}$ implies the
boundedness of the ground state solutions, namely
\begin{equation}
\label{bound}\exists\, c>0 \ \ \mbox{ such that }\ \ \forall\, p\in(2,2_{s}^{*}):
\|u_{p}\|\le c.
\end{equation}

We need now another  technical lemma.
\begin{lemma}\label{limsuptp}
For every $p\in(2,2_{s}^{*})$ let $w_{p}\in \mathcal N_{p}$ be an arbitrary function and assume that
the family 
$\{w_{p}\}_{p\in (2,2^{*}_{s})}$ is bounded in $D^{s,2}_{0}(\Omega)$ and $\{|w_{p}|_{2^{*}_{s}}\}_{p\in (2,2^{*}_{s})}$
is bounded away from zero.

Denote with  $t_{p}>0$ the unique value such that
$t_{p} w_{p}\in\mathcal{N}_{*}$. 
Then
$$\limsup_{p\rightarrow2_{s}^{*}} t_{p}\le1$$
and $\{t_{p}\}_{p\in (2,2^{*}_{s})}$ is bounded away from zero.
In particular by \eqref{bound} and Remark \ref{rem} 
the result follows for the sequence of ground state solutions $\{u_{p}\}_{p\in (2,2^{*}_{s})}$.
\end{lemma}
\begin{proof}
By definition of $\mathcal{N}_{*}, t_{p}$ satisfies
\[
t_{p}^{2_{s}^{*}}|{w}_{p}|_{2_{s}^{*}}^{2_{s}^{*}}=t_{p}^{2}\|{w}_{p}\|^{2}%
\]
and using that ${w}_{p}\in\mathcal{N}_{p}$ and the H\"{o}lder inequality we
get
\begin{equation}
\label{tsup}t_{p}^{2_{s}^{*}-2}=\frac{|{w}_{p}|_{p}^{p}
}
{|{w}_{p}|_{2_{s}^{*}}^{2_{s}^{*}}}
\le\frac{|\Omega|^{\frac{2_{s}^{*}-p}{2_{s}^{*}}
}}{|{w}_{p}|_{2_{s}^{*}}^{2_{s}^{*}-p}}.
\end{equation}
By the embedding $ D^{s,2}_{0}(\Omega)\hookrightarrow L^{2^{*}}(\Omega)$ we 
deduce that the sequence $\{|{w}_{p}|_{2_{s}^{*}}\}_{p\in(2,2_{s}^{*})}$
is bounded.  Since it is also bounded away from zero, the conclusion follows by 
\eqref{tsup}, since
$\lim_{p\rightarrow2_{s}^*}\frac{|\Omega|^{\frac{2_{s}^{*}-p}{2_{s}^{*}}}}{|{w}_{p}|_{2_{s}^{*}}^{2_{s}^{*}-p}}=1\,.$
\end{proof}

\smallskip


%
%
%
Now we can give the main result of this section.

\begin{proposition}
\label{limitmp} For any bounded domain we have
\[
\lim_{p\rightarrow2_{s}^{*}} m_{p}=m_{*}.
\]
\end{proposition}
\begin{proof}
By  Lemma \ref{limitate} it is sufficient to prove that
$$m_{*}\le\liminf_{p\rightarrow2_{s}^{*}}{m}_{p}\,.$$
Let $t_{p}>0$ the unique value such that $t_{p} {u}_{p}\in
\mathcal{N}_{*}$; hence by Lemma \ref{limsuptp} 
$$
m_{*}   \le I_{*}(t_{p}{u}_{p})=\frac{s}{N} t_{p}^{2}\|{u}_{p}\|^{2}
  = {I}_{p}({u}_{p}) t_{p}^{2}+\left( \frac{1}{p}-\frac{1}{2_{s}^{*}
}\right) \|{u}_{p}\|^{2}t_{p}^{2}
  = {m}_{p}\, t_{p}^{2}+o(1)
$$
where $o(1)\rightarrow0$ for $p\rightarrow2_{s}^{*},$ and the  conclusion  follows.
\end{proof}

The following ``global compactness'' result is an extension of the Struwe 
result (see Theorem
3.1 of \cite{Stw}) to the fractional case. Its proof can be found in \cite{PPNA15}; 
see also  \cite{BSY} for the non hilbertian case.


\begin{theorem}\label{St}
Let $\{v_{n}\}$ be a PS sequence for $I_{*}$ (defined in \eqref{eq:I*}) in $D^{s,2}_{0}(\Omega).$
Then there exist a number $k\in\mathbb{N}_{0}$, sequences of points
$\{x_{n}^{j}\} \subset\Omega$ and sequences of radii $\{ R^{j}_{n}\}$ ($1\le
j\le k$) with $R_{n}^{j}\rightarrow+\infty$ for $n\rightarrow+\infty$, there exist a  solution
$v\in D^{s,2}_{0}(\Omega)$ of (\ref{star}) and non trivial solutions
$\{v^{j}\}_{j=1,\ldots, k}\subset D^{s,2}(\mathbb{R}^{N} )$ of
\begin{equation}\label{limit}	
	(-\Delta )^{s}u=|u|^{2_{s}^{*}-2}\ \ \ \text{ in } \mathbb{R}^{N}\,,
\end{equation}
such that, a (relabeled) subsequence $\{v_{n}\}$ satisfies
\begin{eqnarray*}
&v_{n}-v-\sum_{j=1}^{k} v_{R_{n}}^{j}(\cdot-x_{n}^{j}) \rightarrow0 \ \ \text{
in }\ \ {D^{s,2}(\mathbb{R}^{N})}\,,& \\
&I_{*}(v_{n})\rightarrow I_{*}(v)+\sum_{j=1}^{k}\hat{I}(v^{j})&
\end{eqnarray*}
where $\hat{I}: D^{s,2}(\mathbb{R}^{N})\rightarrow\mathbb{R}$ is given by
$$
\hat{I}(u)=\frac{1}{2}\int_{\mathbb R^{2N}} \frac{|u(x) - u(y)|^{2}}{|x-y|^{N+2s}} dxdy-\frac{1}{2_{s}^{*}%
}\int_{\mathbb{R}^{N}}|u|^{2_{s}^{*}}\,dx\,.
$$
\end{theorem}
Roughly speaking, the solutions of \eqref{limit} are responsible for the lack of compactness.

For what concerns $\hat{I}$, it is known (see \cite{CCB}) that it achieves its minimum on
functions of type
\begin{equation}\label{family}	
U_{R}(x-a) = C_{N,s}\left(\frac{R}{R^{2} +|x-a|^{2}} \right)^{\frac{N-2s}{2}}\quad R>0, \ a\in \mathbb R^{N}
\end{equation}
and $C_{N,s}>0$ is a suitable constant.
The  minimum value is exactly 
$$\hat{I}(U_{R}(\cdot-a))=\frac{N}{s}\int_{\mathbb{R}^{N}}|U|^{2^{*}_{s}}dx=m_{*},$$ namely the infimum of $I_{*}$
defined in \eqref{eq:I*} on $\mathcal N_{*}$.
The value of $\hat{I}$ on solutions of \eqref{limit} which do not
belong to the family \eqref{family} 
(which are the unique positive solutions) is greater than $2m_*$
(see \cite[Lemma 2.10]{BSY} for details).
As a consequence, if the sequence $\{v_n\}$ of Theorem \ref{St} is a PS sequence for $I_*$ at level $m_*$,
we deduce  $I_*(v)=0, k=1$ and $v^{1}=U$. 
Moreover, being $v$  a solution of \eqref{star}
and being $I_*$  positive on the solutions,  it follows necessarily that $v=0$;
 so Theorem \ref{St} gives 
$$v_n-U_{R_n}(\cdot-x_n)\rightarrow 0\ \ \ \mbox{in} \ \ D^{s,2}(\mathbb R^N).$$
This ``decomposition'' of the Palais-Smale sequence will be fundamental in the next Section.

%
%

\section{Proof of Theorem \ref{th:th}}	\label{sec:prova}

%
From now on, let  $r>0$ be a radius sufficiently small such that $B_{r}\subset\Omega$ 
(recall $0\in \Omega$) and the sets
$$
\Omega^{+}_{r}=\{x\in\mathbb{R}^{3}:d(x,\Omega)\le r\}\,
$$
$$
\Omega^{-}_{r}=\{x\in\Omega:d(x,\partial\Omega)\ge r\}\,
$$
are homotopically equivalent to $\Omega $. Let also
\begin{equation}\label{h}
h:\Omega^{+}_r \rightarrow\Omega^-_r
\end{equation}
the homotopic equivalence map such that $h|_{\Omega^-_r}$ is the identity. 


\medskip

\subsection{The barycenter map}
We are in a position now to prove the main result of this section.


For $u\in D_{0}^{s,2}(\mathbb R^{N})$ with compact support, 
let us denote with the same symbol $u$ its
trivial extension out of supp\,$u$. The barycenter of $u$ 
is defined as
\begin{equation*}
\beta(u) =\frac{\displaystyle\int_{\mathbb R^{N}} x |u|^{2_{s}^{*}}}{\displaystyle\int_{\mathbb R^{N} }|u|^{2_{s}^{*}}} .
\end{equation*}

We have the following.

\begin{proposition}\label{baricentri}
There exists $\varepsilon>0$ such that if $p\in(2_{s}^{*}-\varepsilon,2_{s}^{*}),$ we have
$$u\in\mathcal{N}_{p} \ \mbox{ and } \ I_{p}(u)<m_{p}+\varepsilon\,
\Longrightarrow\,\beta(u)\in\Omega^{+}_{r}.$$
\end{proposition}
\begin{proof}
We argue by contradiction. Assume that there exist sequences $\varepsilon
_{n}\rightarrow0,p_{n}\rightarrow2_{s}^{*}$ and $w_{n}\in\mathcal{N}_{p_{n}}$ such
that
\begin{equation}
\label{contradiction}I_{p_{n}}(w_{n})\le m_{p_{n}}+\varepsilon_{n}
\  \mbox{ and } \ \ \beta(w_{n})\notin\Omega^{+}_{r}.
\end{equation}
Then, by Proposition \ref{limitmp}
\begin{equation}\label{converg}	
	I_{p_{n}}(w_{n})\rightarrow m_{*}
\end{equation}
and $\{w_n\}$ is bounded in $D_{0}^{s,2}(\Omega)$. Moreover if  $|w_{n}|_{2^{*}_{s}}\to 0$ as $p_{n}\to2^{*}_{s}$
we will have  $I_{p_{n}}(w_{n}) = \frac{p_{n} -2}{2p_{n}}|w_{n}|_{2^{*}_{s}}^{2^{*}_{s}} \to 0$ which is a contradiction.
Then, we are in a position to apply   Lemma
\ref{limsuptp}: let $t_{n}>0$ be such that $t_{n} w_{n}\in\mathcal{N}_{*}$ and we suppose
$t_{n}\rightarrow t_{0}\in (0,1]$. We evaluate
\begin{align*}
I_{p_{n}}(w_{n})-I_{*}(t_{n} w_{n}) & =\left( \frac{1}{2}-\frac{1}{p_{n}
}\right) \|w_{n}\|^{2} 
-\left( \frac{1}{2}-\frac{1}{2_{s}^{*}}\right) t_{n}^{2}
\|w_{n}\|^{2}\\
& = \left( \frac{1}{2}-\frac{1}{p_{n}}\right) \|w_{n}\|^{2}\left( 1-t_{n}
^{2}\right) -\left( \frac{1}{p_{n}}-\frac{1}{2_{s}^{*}}\right) t_{n}^{2}
\|w_{n}\|^{2}\\
&\geq o(1)
\end{align*}
which gives
$
m_{*}\le I_{*}(t_{n} w_{n})  \le I_{p_{n}}(w_{n})+o(1).
$
Then by \eqref{converg}, $I_{*}(t_{n} w_{n})\rightarrow m_{*}$ for $n\rightarrow+\infty.$
The Ekeland's variational principle implies that there exist $\{v_{n}\}\subset
\mathcal{N}_{*} $ and $\{\mu_{n}\}\subset\mathbb{R}$ such that
\begin{align*}
& \|t_{n} w_{n} -v_{n}\|\longrightarrow0\\
& I_{*}(v_{n})=\frac{s}{N}\|v_{n}\|^{2}\longrightarrow m_{*}\\
& I_{*}^{\prime}(v_{n})-\mu_{n} G_{*}^{\prime}(v_{n})\longrightarrow0\nonumber
\end{align*}
and Lemma \ref{lemmaPS}  ensures that $\{v_{n}\}$ is a PS sequence for the free functional
$I_{*}$ at level $m_*$.
By the considerations made  after Theorem \ref{St},
$$
v_{n} - U_{R_{n}}(\cdot-x_{n})\longrightarrow0\ \ \ \text{ in } D^{s,2}(\mathbb{R}^{N})\,,
$$
where $\{x_{n}\}\subset\Omega, R_{n}\rightarrow+\infty$ and we can write
$v_{n}=U_{R_{n}}(\cdot-x_{n})+\zeta_{n}$
with a remainder $\zeta_{n}$ such that $\|\zeta_{n}\|_{D^{s,2}(\mathbb{R}^{N}
)}\rightarrow0 $\,.
It is clear that $t_{n} w_{n}=v_{n}+ t_{n} w_{n}-\zeta_{n}$\,; so, denoting the
remainder again with $\zeta_{n}$, we have
\begin{equation*}\label{eq:tnun}
t_{n} w_{n}=U_{R_{n}}(\cdot-x_{n})+\zeta_{n}\,, \qquad \text{ with } \zeta_{n}\to 0 \ \text{ in } D^{s,2}(\mathbb R^{N}).
\end{equation*}
Adapting the arguments given in \cite[pags. 296-297]{Sicilia} to the present case, after straightforward computations one show that
$$\beta(t_{n}w_{n}) = \beta(w_{n}) = x_{n} +o(1).$$
Since $x_{n}\in \Omega$, we get a  contradiction with \eqref{contradiction}, and this concludes the proof. 
\end{proof}

\subsection{Conclusion of the proof of Theorem \ref{th:th}}\label{subsec:finale}
In the following,
we  add a subscript $r$ ($r>0$ and small as before) to denote the same
quantities defined in the previous sections when the domain
$\Omega$ is replaced by $B_{r}$; namely integrals are taken on $B_{r}$ and norms are taken
for function spaces defined on the ball $B_{r}.$
So
$$
\mathcal{N}_{p,r}=\left\{ u\in D^{s,2}_{0}(B_{r}):\| u\|_{D^{s,2}_{0}(B_{r})}^{2}=|u|^{p}_{L^{p}(B_{r})}\right\}
$$
and, for $u\in\mathcal{N}_{p,r}$ let
$$
I_{p,r}(u)=\frac{p-2}{2p}\|u\|_{D^{s,2}_{0}(B_{r})}^{2}.$$

It is easy to see (for example by means of the Palais Principle of Symmetric Criticality)
 that $\inf_{\mathcal{N}_{p,r}} I_{p,r}$ is achieved on a radially symmetric function $u_{p,r}$, then we have
$$ m_{p,r}=\min_{\mathcal{N}_{p,r}} I_{p,r}=I_{p,r}(u_{p,r})\,.$$
Let
$$
I_{p}^{m_{p,r}}=\left\{ u\in\mathcal{N}_{p}: I_{p}(u)\le m_{p,r}\right\}
$$
be the sublevel of the funcional $I_{p}$ on the Nehari manifold (defined on the domain $\Omega$);
it is non vacuous since $m_{p}<m_{p,r}$.

For $p\in(2,2_{s}^*)$ define the map $\Psi_{p,r}:\Omega_{r}^{-}\rightarrow {\mathcal N}_{p}$ such that
\[
\Psi_{p,r}(y)(x)=
\left\{
\begin{array}
[c]{ccl}
u_{p,r}(\left\vert x-y\right\vert ) & \mbox{if} & x\in B_{r}(y)\\
0 & \mbox{if} & x\in\Omega\setminus B_{r}(y)
\end{array}
\right.
\]
and note that we have
\begin{equation*}
	\beta(\Psi_{p,r}(y))=y\ \ \ \mbox{and}\ \ \ \Psi_{r,p}(y)\in I_p^{m_{p,r}}\,.
\end{equation*}
By Proposition \ref{limitmp} we can write $m_p+k_p=m_{p,r}$ where $k_p>0$ and tends to zero if $p\rightarrow2_{s}^*$. Then
for $\varepsilon>0$ provided by Proposition \ref{baricentri}, there exists a $\bar{p}\in (2,2_{s}^*)$ such
that for every $p\in[\bar{p},2_{s}^*)$ it results $k_p<\varepsilon$, and if $u\in I_p^{m_{p,r}}$ we have
$$I_p(u)\le m_{p,r}<m_p+\varepsilon.$$
Hence the following maps are well-defined (here $h$ is the same as in \eqref{h}):
$$\Omega^-_r\stackrel{\Psi_{p,r}}{\longrightarrow}I_p^{m_{p,r}}\stackrel{h\circ\beta}{\longrightarrow}\Omega_r^-.$$
and the composite map $ h\circ\beta\circ\Psi_{p,r}$ is homotopic to the identity of $\Omega^-_r$. So,  a well known property of the category, gives that
the sublevel $I_p^{m_{p,r}}$ ``dominates'' the set $\Omega_r^-$ in the sense that
$$\mbox{cat}_{I_p^{m_{p,r}}}(I_p^{m_{p,r}})\ge \mbox{cat}_{\Omega^-_r} (\Omega^-_r)$$
(see e.g. \cite{J}) and our choice of $r$ gives 
$\mbox{cat}_{\Omega^-_r} (\Omega^-_r)=\mbox{cat}_{\bar{\Omega}}(\bar{\Omega})$.

Summing up, we have found a sublevel of $I_p$ on $\cal N_p$ with category greater than 
$\mbox{cat}_{\bar{\Omega}}(\bar{\Omega})$. 
Since, as we have already said, the PS condition is verified on $\cal N_p$\,, applying the Lusternik-Schnirel\-mann 
theory we get
the existence of at least $\mbox{cat}_{\bar{\Omega}}(\bar{\Omega})$ critical points for $I_p$
on the manifold $\cal N_p$ which give rise to solutions of \eqref{eq:FBC}.

In case of a non contractible domain $\Omega$, the existence of another solution
can be obtained with the same arguments of \cite{BC2} (see also \cite{Sicilia}).
This is classical by now, but for completeness we recall the proof.
 
 By assumption and by the choice of $r$ it results $\mbox{cat}_{\Omega_{r}^{+}}\,(\Omega_{r}^{-})>1$, that is
$\Omega_{r}^{-}$ is not contractible in $\Omega_{r}^{+}$. 

If now
 the set ${\Psi_{p,r}(\Omega_{r}^{-})}$ were contractible in $I_{p}^{m_{p,r}}$, 
then $\mbox{cat}_{I_{p}^{m_{p,r}}}\,(\Psi_{p,r}(\Omega_{r}^{-}))=1$ and
this means that there exists a map $\mathcal H \in C([0,1]\times{\Psi_{p,r}(\Omega_{r}^{-})}; I_{p}^{m_{p,r}})$ satisfying
$${\mathcal H}(0,u)=u \ \ \forall u\in{\Psi_{p,r}(\Omega_{r}^{-})}\ \  \text{and}$$
$$\exists\, w\in I_{p}^{m_{p,r}}: {\mathcal H}(1,u)=w \ \ \forall u\in {\Psi_{p,r}(\Omega_{r}^{-})} .$$
Then $F=\Psi_{p,r}^{-1}({\Psi_{p,r}(\Omega_{r}^{-})})$ is closed, contains $\Omega_{r}^{-}$ and 
is contractible in $\Omega_{r}^{+}$ as we can see by defining the map 
\begin{equation*}
{\mathcal G}(t,x)=
\begin{cases}
   { \beta(\Psi_{r,p}(x))} & \text{if \ \ $0\le t\le 1/2$}, \\ 
    \beta ({\mathcal H}(2t-1, \Psi_{p,r}(x))) &\text{if  \ \ $1/2\le t\le1$}.
\end{cases}
\end{equation*}
Then also $\Omega_{r}^{-}$ would be contractible in $\Omega_{r}^{+}$ giving a contraddiction.

On the other hand we can choose a function $z\in {\mathcal N}_{p}\setminus{\Psi_{p,r}(\Omega_{r}^{-})}$ so that the cone 
$$\mathcal C=\left\lbrace \theta z+(1-\theta) u : u\in {\Psi_{p,r}(\Omega_{r}^{-})}, \theta\in [0,1]\right\rbrace $$
is compact and contractible in $D^{s,2}_{0}(\Omega)$ and  $0\notin{\mathcal C}$. 
Denoting with $t_{u}$ the unique positive number provided by Lemma \ref{lemmanehari}, it follows that if we set
$$\hat{\mathcal C}:=\{t_{u}u : u\in \mathcal C\}, \ \ \ M_{p}:=\max_{\hat{\mathcal C}} I_{p}$$
then $\hat{\mathcal C}$ is contractible in $I_{p}^{M_{p}}$ and $M_p>m_{p,r}.$ As a consequence
also ${\Psi_{p,r}(\Omega_{r}^{-})}$ is contractible in $I_{p}^{M_{p}}.$ 

Resuming, the set $\Psi_{p,r}(\Omega_{r}^{-}) $ is contractible in $I_{p}^{M_{p}}$ and not in $I_{p}^{m_{p,r}}$. 
Since the PS condition is satisfied we deduce the existence of another critical point
with critical level between $m_{p,r}$ and $M_{p}$.

\medskip

That the solutions we have found are positive, is a simple consequence of the fact that
 we can apply all the previous
machinery replacing the functional \eqref{eq:Ip} with
\begin{equation*}
I_{p}(u)=\frac{1}{2}\int_{\mathbb R^{2N}}\frac{|u(x)-u(y)|^{2}}{|x-y|^{N+2s}} dx dy+\frac12\int_{\mathbb R^{N}}u^{2} -\frac1p\int_{\mathbb R^{N}}(u^{+})^{p}\,,
  \qquad u\in D_{0}^{s,2}(\Omega)
\end{equation*}
and then by the Maximum Principle we conclude.

\end{document}